\documentclass[11pt]{article}

\usepackage[letterpaper, margin=1in]{geometry}

\usepackage{times}

\usepackage{amsmath,amssymb}
\usepackage{verbatim}
\usepackage{amsthm}
\usepackage{amsfonts}
\usepackage{xspace}
\usepackage{subcaption}
\usepackage{setspace}
\usepackage{url}
\usepackage{mathrsfs}
\usepackage{hyperref}

\usepackage{graphicx}% Required for inserting images
\usepackage{xcolor}

\theoremstyle{plain}
\newtheorem{theorem}{Theorem}[section]

\newtheorem{lemma}[theorem]{Lemma}
\newtheorem{corollary}[theorem]{Corollary}

\newtheorem*{claim*}{Claim}

\theoremstyle{definition}
\newtheorem*{definition*}{Definition}

\theoremstyle{remark}
\newtheorem{remark}[theorem]{Remark}

\numberwithin{equation}{section}

\def \one {\mathsf{1}}

\def \sbs {\subseteq}

\def \D {\mathcal{D}}
\def \R {\mathbb{R}}

\def \E {\mathbb{E}}

\def \N {\mathbb{N}}
\def \P {\mathbb{P}}
\def \mcS {\mathcal{S}}

\def \Z {\mathbb{Z}}
\def \d {\mathrm{d}}

\DeclareMathOperator{\Lip}{Lip}
\DeclareMathOperator{\TC}{TC}

\DeclareMathOperator{\Per}{Per}
\DeclareMathOperator{\Vol}{Vol}
\DeclareMathOperator{\supp}{supp}
\DeclareMathOperator{\dist}{dist}
\DeclareMathOperator{\diam}{diam}
\DeclareMathOperator{\dVol}{dVol}
\newcommand{\tc}{{\rm TC}\hskip0.02cm}
\newcommand{\emd}{{\rm EMD}\hskip0.02cm}
\newcommand{\conv}{{\rm conv}\hskip0.02cm}

\title{$L_1$-Distortion of Earth Mover Distances and Transportation Cost Spaces on High Dimensional Grids
}

\author{Chris Gartland\thanks{C. Gartland was supported by the National Science Foundation under Grant Number DMS-2555144.} \and
Mikhail Ostrovskii \and
Yuval Rabani\thanks{Y. Rabani was supported by ISF grant 389-22 and by BSF grant 2023607.} \and
Robert Young
}

\date{}

\begin{document}

% The article template restarts page numbering after the contents. Disable
% front-matter page anchors so that the restarted numbers stay link-safe.
\hypersetup{pageanchor=false}

\maketitle

\begin{abstract}
We prove that the distortion of any embedding into $L_1$ of the transportation cost space or earth mover distance over a $d$-dimensional grid $\{1,\dots m\}^d$ is $\Omega(\log N)$, where $N$ is the number of vertices and the implicit constant is universal (in particular, independent of dimension). This lower bound matches the universal upper bound $O(\log N)$ holding for any $N$-point metric space. Our proof relies on a new Poincaré inequality for real-valued functions on the grid, based on random measures supported on dyadic cubes.
\end{abstract}

\par


\tableofcontents

\thispagestyle{empty}
\newpage
\setcounter{page}{1}
\hypersetup{pageanchor=true}

\section{Introduction}

Optimal-transport based distances on spaces of measures on metric spaces have a rich history, dating back to the classical work of Monge and Kantorovich (see \cite{San15,Vil09} for historical information), and appear in many different contexts. For example, they appear in studies of (1) similarity measurement between probability measures, (2) Banach spaces of transportation problems with the norm corresponding to optimal transportation cost, and (3) free linear spaces generated by metric spaces and endowed with the norms induced by these metric spaces. This led to a large variety of terminology meaning the same or closely related notions, with such names as Arens-Eells space, earth mover distance, Gini distance, Kantorovich distance, Kantorovich-Rubinstein norm, Lipschitz-free Banach space, transportation cost norm, Wasserstein distance, and Wasserstein space. In this article, we use the terms {\it earth mover distance} ($\emd$) to refer to the metric on probability measures on a metric space $X$ and {\it transportation cost norm} to refer to the norm on the linear space of signed measures on $X$ with total measure 0 (see \S\ref{sec:results} and \S\ref{sec:background} for the precise definitions).

In computer science, $\emd$ has ample applications in comparison of complex data objects. For instance, it has been recognized to capture well the differences between images \cite{RTG98,RTG00,Cha02,IT03} and between text documents \cite{KSKW15,RCP16,YCCMS19}. As high-dimensional spaces are often used to describe documents and even two-dimensional color images, the study of $\emd$ on high-dimensional spaces, in particular, on high-dimensional grids, is a popular direction \cite{AFPVX17,AIK08,BCJW25, CCRW23,CJLW22,FL23,IT03}. Low-distortion embeddings into $L_1$ (see \S\ref{sec:results} for the definition) are one of the standard tools in this study~\cite{AIK08,Cha02,IT03}. They are of particular importance in the approximate nearest neighbor (ANN) search problem. Indeed, a typical strategy towards building an ANN data structure for $\emd$ is to first embed it into $L_1$ with low distortion, and then utilize technology developed for $L_1$ (often locality-sensitive hashing \cite{IT03},\cite[Corollary~3.3]{AIK08},\cite[\S7.1.1]{JWZ24}) to perform the nearest neighbor search there. A similar motivation arises in control theory, where often the state space is represented as a multi-dimensional grid, each coordinate reflecting the discretization of a monitored numerical parameter. $\emd$ is suggested in various contexts as a robust measure of the difference between uncertain states (e.g.~\cite{Yan17,LKKB17}), or states of swarms of controlled devices (e.g.~\cite{MHYZ23}), and low-distortion embedding into $L_1$ is similarly attractive to reduce computational problems to known solutions.
Due to these connections, it is essential to clarify the level of distortion needed to embed $\emd$ over a high-dimensional grid into $L_1$. This work provides such clarification.

\subsection{Our results} \label{sec:results}

Let $m,d\in\N$. Let $G=(V,E)$ be the $d$-dimensional grid graph of side length  $m-1$; i.e., $G$ is the $d$-fold graph Cartesian product $$P_m\square P_m\square\cdots\square P_m,$$ where $P_m$ is the path graph on $m$ nodes. Let $\rho_G$ denote the shortest path metric on $V$. Denote $N = m^d = |V|$, where $|A|$ denotes the cardinality of a set $A$. We consider the Banach space $\tc(G)$ of signed measures on $V$ of total measure $0$, endowed with the $L_1$ optimal transport norm $\|\cdot\|_{\tc}$ (characterized by its unit ball $\conv\left\{\frac{\delta_x - \delta_y}{\rho_G(x,y)}\colon x,y\in V, x\ne y\right\}$). Recall that the {\it $L_1$-distortion} of a metric space $(X,\rho_X)$ is the quantity $c_1(X)$ defined to be the infimal $C \geq 1$ for which there exists a map $X \to L_1([0,1])$ with bi-Lipschitz distortion $C$, where the {\it bi-Lipschitz distortion} of a map $F: X\to Y$ between metric spaces is the least constant $L$ for which there exists a scaling factor $s\in(0,\infty)$ with
$$\rho_X(x,y) \leq s\rho_Y(F(x),F(y)) \leq L\rho_X(x,y)$$
for all $x,y\in X$. Note that when $X$ or $Y$ is a normed space, the scaling factor $s$ can always be taken to be 1. The following theorem is our main result.

\begin{theorem}\label{thm:mainthm}
If $d \geq 2$, then any embedding of $\tc(G)$ into $L_1$ has bi-Lipschitz distortion $\Omega(\log N) = \Omega(d\log m)$. 
\end{theorem}

We also consider the metric space $\emd(G)$ consisting of probability measures on $V$ equipped with distance $\emd(\mu,\nu) = \|\mu-\nu\|_{\tc}$. For any finite metric space $X$, the $L_1$-distortion of the metric space $\emd(X)$ coincides with the linear $L_1$-distortion of the Banach space $\tc(X)$ (see Theorem~\ref{T:lin&EMD} for the precise statement). This important reduction to linear maps was observed in \cite[Lemma~3.1]{NS07} for $2$-dimensional grids, and the generalization for an arbitrary finite metric space is presented in \cite{GO26}. Hence, the next corollary follows immediately from Theorems~\ref{thm:mainthm} and \ref{T:lin&EMD}.

\begin{corollary}\label{cor:maincor}
If $d \geq 2$, then any embedding of $\emd(G)$ into $L_1$ has bi-Lipschitz distortion $\Omega(\log N) = \Omega(d\log m)$. 
\end{corollary}

We remark that when $d=1$, the path graph $G$ isometrically embeds into $\R$, and it is classical that $c_1(\emd(G)) \leq c_1(\emd(\R)) = 1$ (see, for example, \cite[Proposition~2.17]{San15}).

\subsection{Relation to past work} \label{sec:past}

Our lower bound matches asymptotically the universal upper bound of $O(\log N)$ that applies for any $N$-point metric space, through stochastic embeddings into dominating tree metrics, due to Fakcharoenphol, Rao, and Talwar~\cite{FRT04}. It is an intriguing question if there are techniques for embedding $\tc(X)$ spaces into $L_1$ that are genuinely different from the ones induced by stochastically embedding $X$ into dominating tree metrics. In particular, an important open question is whether any technique can yield distortion of a lower order than the stochastic embedding distortion of $X$ into dominating tree metrics. Theorem~\ref{thm:mainthm} eliminates the high-dimensional grid as a counterexample candidate.

Our results build upon several prior works. Khot and Naor~\cite[Corollary~2]{KN06} established a lower bound of $\Omega(d)$ for the case of $m=2$ (the binary cube), and recently Gartland and Ostrovskii~\cite{GO26} gave a lower bound of $\Omega(\log m)$ for the case of $d = 2$ (the planar grid). As $G$ isometrically contains both a $d$-dimensional binary cube and an $m\times m$ planar grid, these results jointly imply a lower bound of
\begin{equation} \label{eq:lowerbound}
    c_1(\tc(G)) = \Omega(\max\{d,\log m\}),
\end{equation}
which Theorem~\ref{thm:mainthm} improves upon. Work on lower bounds of the $L_1$-distortion of the transportation cost space over the diamond graphs has also been conducted, first in \cite{BGS23}, and then with sharp bounds obtained in \cite{GO26}. The latter work also rules out the diamond graphs as potential counterexamples to the question asked in the previous paragraph.

Our proof of Theorem~\ref{thm:mainthm} uses a new Poincaré-type inequality \eqref{eq:intro-Poincare} based on random measures supported on dyadic cubes, whose proof applies large deviation bounds and Bollob\'as-Leader's (edge) isoperimetric properties of the high-dimensional grid \cite{BL91}. Both the construction of these random measures and the analysis of the inequality are novel, diverging from all previous work on the problem of $L_1$-distortion of transportation cost spaces (e.g., \cite{KN06,NS07,BGS23,GO26}).

Interestingly, our probabilistic proof of the Poincaré inequality \eqref{eq:intro-Poincare} requires dimension $d > 2$ and does not seem to be applicable to $d=2$, due to insufficient large deviation bounds, required in Lemmas~\ref{lem:mu(grid)} and \ref{lem:Poincare1-alt}, in this dimension. Conversely, the Poincaré inequality of \cite{GO26} relies on more geometrically careful constructions and analysis, which appear to be specific to dimension 2 (for example, certain metric-topological properties of subsets $A \sbs \{1,\dots n\}^2$ are exploited, such as $\partial A$ decomposing as a disjoint union of simple closed curves, and $\min\{\diam(A^c),\diam(A)\} \lesssim |\partial A|$ when $A$ and its complement $A^c$ are connected). It is not clear what the corresponding geometric constructions should be in higher dimensions, if any exist at all. Therefore, our proof of Theorem~\ref{thm:mainthm} is new for $d > 2$, and its proof for $d=2$ relies on \cite{GO26}.

Khot and Naor's proof of \cite[Corollary~2]{KN06} also relies on a Poincaré-type inequality \cite[Lemma~1]{KN06}, but for functions on quotients of the binary cube $\{0,1\}^d$ by special subgroups $C^\perp$ (which are known to exist by coding theory). The proof of their Poincaré inequality uses Fourier analysis on the abelian group $\{0,1\}^d = (\Z/2\Z)^d$.
They use the Poincaré inequality to show that the quotient space $\{0,1\}^d / C^\perp$ has $L_1$-distortion $\Omega(d)$ \cite[Theorem~1]{KN06}, and they observe that $\{0,1\}^d / C^\perp$ admits an isometric embedding into $\emd(\{0,1\}^d)$ \cite[Lemma~5]{KN06}, hence proving their result. An important difference between this proof and our proof of Theorem~\ref{thm:mainthm} is that our Poincaré inequality \eqref{eq:intro-Poincare} is a {\it multiscale inequality}, as the left-hand side involves a sum, over $k$, of distances on the order $2^kd$ (see \S\ref{sec:outline} for more discussion on this distinction). This is in contrast with the Poincaré inequality of Khot-Naor, which is only single-scale (all distances on the left-hand side of their inequality are of order $d$). Moreover, the proof of our Poincaré inequality \eqref{eq:intro-Poincare} is novel and has a very different flavor, using neither Fourier analytic nor group theoretic tools, but rather relies only on basic ingredients like the aforementioned high-dimensional isoperimetric inequalities and large deviation bounds.

Theorem~\ref{thm:mainthm} in the case $d > 2$ is Corollary~\ref{cor:L1distortion}, proved in \S\ref{sec:Poincare-distortion} (the remaining case $d=2$ is covered by the lower bound \eqref{eq:lowerbound}). We remark that the case $m=2$ of Theorem~\ref{thm:mainthm} gives a new and (in some aspects) simpler proof of the Khot-Naor bound $c_1(\TC(\{0,1\}^d)) = \Omega(d)$.

\subsection{Further history}

We provide in this subsection additional information on the history of the $L_1$-distortion of transportation cost spaces over finite metric spaces.

% Although we defined the transportation cost space only over finite metric spaces, the definition can be extended in a straighforward way to infinite ones, and much of the early work was done in this setting. In \cite{Cie60}, Ciesielski proved that the space of Lipschitz functions on $[0,1]$ equipped with the snowflake metric $d(x,y) = |x-y|^\alpha$ is isomorphic to $\ell_\infty$ for every $0<\alpha<1$ (from this, one can obtain using duality that $\tc([0,1],|\cdot|^\alpha)$ is isomorphic to $\ell_1$). It was recorded in both \cite[Assertion~6]{Kad85} and \cite[Remark~2]{Pes86} that $\tc([0,1])$ is isomorphic to $L_1([0,1])$, and the question of when $\tc(X)$ is isomorphic to $L_1(\mu)$ for some measure $\mu$ appeared in the latter.

To the best of our knowledge, the first result on lower bounds appeared in a paper of Bourgain \cite[Appendix, p.~227]{Bou86}, where he proved that the transportation cost spaces $\tc(\{0,1\}^n)$ on binary cubes have subspaces uniformly isomorphic to $\ell_\infty^{d_n}$, where $d_n\to\infty$ as $n\to\infty$. This, by well-known results of Banach space geometry, prevents their linear embeddings into $L_1$ with uniformly bounded distortions. At approximately the same time, Bourgain \cite{Bou87} developed a discretization technique which transferred lower bounds on distortion of linear embeddings to nonlinear embeddings. At the time he was working, the aforementioned applications of $L_1$-embeddings to theoretical computer science were not known, and Bourgain's work had different motivation. Indeed, it was an open problem in nonlinear Banach space theory to determine an appropriate definition of cotype for metric spaces that, when restricted to Banach spaces, coincided with Rademacher cotype. One proposal was to declare $X$ to have metric cotype $q$ if the transportation cost space $\tc(X)$ has cotype $q$. Bourgain proved his theorem to show that such a definition was not suitable, as it implies $\ell_1$ has infinite metric cotype in this sense, while having Rademacher cotype 2 (a suitable notion of metric cotype was later discovered by Mendel and Naor \cite{MN08}).

We note that Bourgain's result alone cannot reach the optimal lower bound of $\Omega(n)$ for the $L_1$-distortion of $\tc(\{0,1\}^n)$. The reason is because whenever $\tc(X)$ contains $\ell_\infty^d$ with distortion $O(1)$, it holds that $d = O(\log|X|)$ (due to the fact that $\tc(X)$ is a quotient of $\ell_1^{|X|^2}$), while the $L_1$-distortion of $\ell_\infty^d$ is at most $O(\sqrt{d})$.

For upper bounds on the $L_1$-distortion of $\emd$, Charikar observed in \cite{Cha02} that a formula of Kleinberg and Tardos \cite{KT99} yields an isometric embedding of $\emd$ over tree metrics into $\ell_1$. Furthermore, he explained in that work how the probabilistic approximation of finite metric spaces by trees defined by Bartal in \cite{Bar96,Bar98} yields a universal upper bound of $O(\log n\log\log n)$ for the $L_1$-distortion of $\emd(X)$ over arbitrary $n$-point metric spaces $X$. This bound was later improved to the best possible of $O(\log n)$ by Fakcharoenphol, Rao, and Talwar \cite{FRT04}.

After these publications, estimates of $c_1(\emd(X))$ attracted significant attention. First came the previously discussed lower bound $c_1(\emd(\{0,1\}^n)) = \Omega(n)$ of Khot and Naor \cite{KN06}, yielding (tight) quantitative bounds not addressed by \cite{Bou87}. Next, Naor and Schechtman proved that $c_1(\tc(\{1,\dots n\}^2])) = \Omega(\sqrt{\log n})$ by discretizing a result of Kislyakov \cite[Theorem 3]{Kis75}, who solved a similar $L_1$-embedding problem, but in the infinite-dimensional setting. With some slight modifications, Kislyakov's proof can be seen to consist of the following steps: use duality to convert a linear isomorphic embedding of $\tc([0,1]^2)$ into $L_1([0,1])$ into a quotient map from $L_\infty([0,1])$ onto the space of Lipschitz functions $\Lip([0,1]^2)$, postcompose this quotient map with three operators well known in analysis (including the Sobolev embedding and Fourier coefficient maps), and obtain a contradiction by using several important results from different branches of Banach space theory. This strategy was generalized and adapted to the diamond graphs $\{D_n\}_n$ in \cite{BGS23}, and the lower bound $c_1(\tc(D_n))=\Omega(\sqrt{n})=\Omega(\sqrt{|D_n|})$ was obtained in that paper. The authors of \cite{BGS23} were motivated not only by \cite{NS07}, but also by the paper \cite{DKO20}, where it was proved that the Banach-Mazur distance of $\tc(D_n)$ to $\ell_1^{|D_n|-1}$ is $\Theta(n)$, and where the problem to estimate $c_1(\tc(D_n))$ was suggested.

Finally, the previously discussed recent work \cite{GO26} simultaneously simplified and strengthened the adapatations of \cite{NS07} and \cite{BGS23} of Kislyakov's argument \cite{Kis75}. One of the main differences was the usage of a customized Sobolev inequality with mixed exponents, which is responsible for getting rid of the square-root in the $L_1$-distortion lower bounds of $\tc(\{1,\dots n\}^2)$ and $\tc(D_n)$. Furthermore, this new method was sufficiently flexible to prove sharp $L_1$-distortion bounds for the transportation cost space over more general slash powers like Laakso graphs. Prior to this, no method for proving any nontrivial lower bounds were known.

In the course of work on this project, we considered each of the proof techniques described above: cotype lower bounds as in \cite{Bou87}, Fourier analysis as in \cite{KN06}, and customized Sobolev inequalities with mixed exponents as in \cite{GO26}. However, these techniques did not seem to work in the current setting, leading us to develop a method described in \S\ref{sec:outline}.

\subsection{Notation and background} \label{sec:background}

Let $n,d \in \N$. Throughout the remainder of the article, we {\bf assume that} $\boldsymbol{d \geq 3}$. It will be convenient for us to consider the $d$-dimensional grid as an integer grid in $\mathbb{R}^d$ with the metric induced by the $\ell_1$-norm, and it will also be convenient to assume that sides of the grids are powers of $2$. Specifically, denote the set of integers $\{1,\dots 2^n\}$ by $[2^n]$ and its $d$-fold Cartesian product by $[2^n]^d$. We equip $[2^n]^d$ with edge set
\begin{equation*}
E([2^n]^d) := \{\{u,v\} \sbs V: u = v \pm e_i \text{ for some } i\in\{1,\dots d\}\},
\end{equation*}
where $\{e_i\}_{i=1}^d$ is the standard basis of $\R^d$. Equipped with the (unweighted) shortest path metric, the inclusion of $[2^n]^d$ into $\R^d$ is isometric when $\R^d$ is equipped with the $\ell_1$-norm.
 
Let $(X,\rho)$ be a finite metric space. Consider a real-valued measure $\mu$ on $X$ satisfying $\mu(X)=0$. A natural and important interpretation of such a real measure is as a {\it transportation problem}: one needs to transport a certain product from points where $\mu(v)>0$ to points where $\mu(v)<0$. 
 
One can easily see that a transportation problem $\mu$ can be represented as
\begin{equation}\label{E:TranspPlan}
    \mu=a_1(\delta_{x_1}-\delta_{y_1})+a_2(\delta_{x_2}-\delta_{y_2})+\dots+a_n(\delta_{x_n}-\delta_{y_n}),
\end{equation}
where $a_i\ge 0$, $x_i,y_i\in X$, and $\delta_u$ for $u\in X$ is the unit measure at $u$. We call each such representation a {\it transportation plan} for $\mu$. The {\it cost} of the transportation plan \eqref{E:TranspPlan} is defined as $\sum_{i=1}^n a_i\rho(x_i,y_i)$.
 
The {\it transportation cost norm} $\|\mu\|_{\tc}$ on the linear space of all transportation problems $\mu$ on $X$ is the infimum of costs of transportation plans
satisfying \eqref{E:TranspPlan}. We get a finite-dimensional Banach space, which we denote $\tc(X)$. As mentioned in \S\ref{sec:results}, $\|\cdot\|_{\tc}$ is the unique norm on transportation problems whose unit ball is $\conv\left\{\frac{\delta_x - \delta_y}{\rho(x,y)}\colon x,y\in X, x\ne y\right\}$. Alternatively, Kantorovich duality \cite[Particular Case~5.16]{Vil09} gives the characterization $\|\mu\|_{\tc} = \sup\{|\int_X fd\mu|: f: X \to \R \text{ is 1-Lipschitz}\}$.

\subsection{Proof outline and organization of paper} \label{sec:outline}

The prototypical method used to prove that a metric space $(X,\rho)$ has $L_1$-distortion at least $D$ is to prove that there are positive constants $\{a_i\}_i, \{b_j\}_j \sbs (0,\infty)$, and pairs of points $\{(x_i,y_i)\}_i,\{(u_j,v_j)\}_j \sbs X \times X$ such that
$$\sum_i a_i \rho(x_i,y_i) \geq D \sum_j b_j \rho(u_j,v_j),$$
while for any map $f: X \to L_1(0,1)$,
$$\sum_i a_i \|f(x_i)-f(y_i)\|_1 \leq \sum_j b_j \|f(u_j)-f(v_j)\|_1$$
(see, for example, \cite[Proposition~4.3]{Ost13}). Inequalities of the above type for $L_1$-valued functions are often known as {\it Poincaré inequalities}. The proof of Theorem~\ref{thm:mainthm}, which we outline next, can be viewed as fitting within this framework.

First, by standard techniques, we can reduce Theorem~\ref{thm:mainthm} to an estimate of the distortion of a linear embedding $F:\tc([2^n]^d)\to \ell_1^{M}$, for large but finite $M$
(see Theorem~\ref{T:lin&EMD}). We prove the theorem by showing that if $F$ is non-contractive on a certain family of random measures $\nu_k$ on $[2^n]^d$, then the operator norm of $F$ satisfies\footnote{We use the following asymptotic notation. Given $a,b>0$, by writing
  $a\lesssim b$ or $b\gtrsim a$ we mean that $a\le \omega b$ for some universal constant $\omega>0$. We write $a \approx b$ to mean $a \lesssim b \lesssim a$.} $\|F\|_{op} \gtrsim \log 2^{nd}$.

For any $k\in \{1,\dots,n\}$, we consider the vertex set $[2^n]^d$ as a disjoint union of vertex sets of subcubes of side length $2^k$ and let $\nu_k$ be a sum of signed uniform measures on a random subset of these subcubes, plus a small correction to ensure that it has total mass zero. We choose the distribution so that $\nu_k$ has expected transportation cost $\E\|\nu_k\|_{\tc} \approx d2^n$ for all $24\le k\le n$.
Crucially, we show that for all $f: [2^n]^d \to \R$, the $\nu_k$'s satisfy a Poincaré-type inequality of the form 
\begin{equation}\label{eq:intro-Poincare}
\notag  \sum_{k=24}^n \E| \int f \d \nu_k|
  \lesssim \frac{1}{d2^{n(d-1)}} \sum_{\{u,v\}\in E([2^n]^d)} |f(v)-f(u)|.
\end{equation}

The main ingredients in the proof of this inequality are high-dimensional isoperimetric inequalities on grids, established in \cite{BL91}, and large deviation inequalities for i.i.d. sums.

We then estimate the distortion of $F$ using its coordinate functions $f_1,\dots, f_N: [2^n]^d \to \R$. These are the functions such that for any measure $\mu \in \tc([2^n]^d)$, 
\begin{equation*}
F(\mu) = \left(\int f_1 \d \mu, \dots, \int f_N \d \mu\right) \text{\qquad\qquad}
\|F(\mu)\|_1 = \sum_i \left|\int f_i \d \mu\right|.
\end{equation*}
If we apply \eqref{eq:intro-Poincare} to the coordinate functions of $F$ and sum, we obtain
\begin{align*}
\sum_{k=24}^n \E\|F(\nu_k)\|_1 &= \sum_i \sum_{k=24}^n \E\left| \int f_i \d \nu_k\right| \lesssim \sum_i\frac{1}{d2^{n(d-1)}} \!\!\!\sum_{\{u,v\}\in E([2^n]^d)}\!\!\! |f_i(v)-f_i(u)| \\
=& \frac{1}{d2^{n(d-1)}} \!\!\!\sum_{\{u,v\}\in E([2^n]^d)}\!\!\! \|F(\delta_v)-F(\delta_u)\|_1 \le 2^n \|F\|_{op}.
\end{align*}
If $F$ is non-contractive, then
$$2^n \|F\|_{op} \gtrsim \sum_{k=24}^n \E\|F(\nu_k)\|_1 \ge \sum_{k=24}^n d2^n$$
and thus $\|F\|_{op} \gtrsim (n-23)d$. This is the desired distortion estimate.

We emphasize here that, using the argument of the preceding paragraph, each individual single-scale inequality $\E\left| \int f \d \nu_k\right| \lesssim \frac{1}{d2^{n(d-1)}} \sum_{\{u,v\}\in E([2^n]^d)} |f(v)-f(u)|$
is sufficient by itself to yield $c_1(\TC([2^n]^d)) = \Omega(d)$, and this is analogous to the inequality used by Khot-Naor in \cite{KN06} to obtain $c_1(\TC([2]^d)) = \Omega(d)$. Hence, the main novelty of our inequality \eqref{eq:intro-Poincare} is that the left-hand side is a sum of scales instead of a single scale, yielding the larger distortion bound of $\Omega(nd)$.

The measures $\nu_k$ are constructed in \S\ref{sec:construction}, and the lower bound on $\E[\|\nu_k\|_{\TC}]$ is established in \S\ref{sec:lower-TC}. In \S\ref{sec:proof-Poincare}, we prove \eqref{eq:intro-Poincare}, and in \S\ref{sec:Poincare-distortion}, we prove Corollary~\ref{cor:L1distortion} (which, together with \eqref{eq:lowerbound}, implies Theorem~\ref{thm:mainthm}).

\section{Construction of random measures} \label{sec:construction}

In this section, we construct the random measures $\nu_k$ for $k\in \{1,\dots,n\}$. For $A \sbs [2^n]^d$, define
\begin{equation} \label{E:Vol}
    \Vol(A) := \frac{|A|}{|[2^n]^d|} = 2^{-nd}|A|.
\end{equation}

Fix $k \in \{1,\dots n\}$. Let $\D_k$ denote the collection of $2^{kd}$ dyadic cubes $Q$ in $[2^n]^d$ whose sides have cardinality $2^{n-k}$. More precisely, for $j \in \{1,\dots 2^{k}\}$, define
\begin{equation*}
    I_{j}^{(k)} := \{1+(j-1)2^{n-k},\dots, j2^{n-k}\}.
\end{equation*}
so that $\{I^{(k)}_j\}_{j=1}^{2^k}$ partitions $[2^n]$ into paths of cardinality $2^{n-k}$. Then we form the product partition of $[2^n]^d$:
\begin{equation*}
    \D_k := \{Q \sbs [2^n]^d: Q = \Pi_{i=1}^d I^{(k)}_{j_i} \text{ for some }(j_i)_{i=1}^d \in\{1,\dots 2^{k}\}^d\}.
\end{equation*}
Note that $\Vol(Q) = 2^{-kd}$ for every $Q \in \D_k$.

Recall that $d\in\N$ with $d\ge 3$. For the remainder of the article, set $\boldsymbol{p:= 2^{-4d}}$, and let $\{X_Q\}_{Q\in\D_k}$ be a collection of i.i.d.\ random variables with $\P(X_Q = \pm 1) = p/2$ and $\P(X_Q=0)=1-p$. For each $Q \in \D_k$, define the measure $\Vol_Q$ on $[2^n]^d$ by $\Vol_Q(A) := \Vol(A \cap Q)$. Define a random signed measure $\mu_{k}$ on $[2^n]^d$ by
\begin{equation*}
    \mu_k := \sum_{Q\in\D_k}p^{-1}2^k X_Q \Vol_Q.
\end{equation*}
While the total mass of $\mu_k$ will typically be small, we need a random measure whose total mass is exactly 0 almost surely. The most pragmatic way to do this seems to be to subtract a measure that has constant density with respect to $\Vol$. Towards this end, we define the random signed measure $\nu_k$ on $[2^n]^d$ by
\begin{equation}\label{E:DefNu}
    \nu_k := \mu_k - \mu_k([2^n]^d)\Vol.
\end{equation}
Since $\nu_k$ has 0 total mass almost surely, it is a $\TC([2^n]^d)$-valued random variable. The next lemma bounds the expected total variation norm of the constant-density part of $\nu_k$, which will allow us to reduce all relevant estimates needed for $\nu_k$ to estimates for $\mu_k$. It is the first place where the requirement $d \geq 3$ comes into play.

\begin{lemma} \label{lem:mu(grid)}
Let $k \in \N$ with $k \geq 5$. Then $\E[|\mu_k([2^n]^d)|] \leq (\sqrt{2})^{15-k-d} < \frac{1}{d}(\sqrt{2})^{16-k}$.
\end{lemma}

\begin{proof}
We begin by bounding the $L_1$-norm of $\mu_k([2^n]^d)$ by its $L_2$-norm, then computing the $L_2$-norm:
\begin{align*}
  \E[|\mu_k([2^n]^d)|]^2 &\leq \E[|\mu_k([2^n]^d)|^2] \\
    &= p^{-2}2^{2k}\sum_{Q\in\D_k}\E[X_Q^2]\Vol_Q([2^n]^d)^2 \\
    &= p^{-2}2^{2k}  p \sum_{Q\in\D_k}\Vol(Q)^2 \\
    &= p^{-1}2^{2k} 2^{kd}2^{-2kd} \\
    &= 2^{4d-kd+2k}.
\end{align*}
Since $d \ge 3$ and $k\ge 5$, we have
$$4d - kd + 2k = 15 - k - d + (3-d)(k-5) \le 15 - k - d.$$
This proves the first inequality. For the second inequality, we begin with
$$(\sqrt{2})^{15 - d - k} = \frac{1}{d} (\sqrt{2})^{16 - k} \cdot d (\sqrt{2})^{-d-1},$$
and then observe that
$$d (\sqrt{2})^{-d-1} = \sqrt{d^2 2^{-d-1}} < 1$$
because $d\ge 3$, as desired.
\end{proof}

\begin{remark}
We wish to remark on the choice of the value $p = 2^{-4d}$. In the proof of Lemma~\ref{lem:mu(grid)}, it is clear that we could have chosen $p = 2^{-cd}$ for any constant $c \geq 0$ and obtained a similar inequality, although a larger choice of $c$ would come with the cost of increasing the constant on the right-hand side and increasing the value of $k_0$ in the requirement $k \geq k_0$. Such costs are benign and would at most affect the implicit constant in Theorem~\ref{thm:mainthm} (note that a choice of $p$ decaying super-exponentially as a function of $d$ would break the proof of Lemma~\ref{lem:mu(grid)} altogether, though).  Similar statements are true concerning the effect of the choice of $p$ on the inequality of Lemma~\ref{lem:Poincare1-alt}, which is the other lemma in the article whose proof employs large deviation bounds. On the other hand, choosing the constant $c$ in $p=2^{-cd}$ too small would cause a fatal error in the lower bound of Theorem~\ref{thm:TCnorm}, as explained in footnote~\ref{foot:p=2^(-4d)} on page~\pageref{foot:p=2^(-4d)}. We make the choice $c = 4$ to simplify numerics, as it is the least integer that is large enough for the proof of Theorem~\ref{thm:TCnorm} to go through.
\end{remark}

\section{Lower bound on expected transportation cost norm} \label{sec:lower-TC}
Next, we bound the expected transportation cost of $\nu_k$ from below by showing that there is a $c>0$ such that most of the measure of $\nu_k$ has to travel at least distance $cd 2^{n-k}$. This distance comes from the following geometric lemma.
If $A,B$ are two nonempty subsets of a finite metric space $(X,\rho)$, then we set $\displaystyle \dist(A,B) := \min_{a\in A,b\in B}\rho(a,b)$.

\begin{lemma} \label{lem:cardinality}
Let $Q_0\in \D_k$ and let $c > 0$. Then
$$|\{Q\in \D_k : \dist(Q_0, Q) \le c d 2^{n-k}\}| \le (2e (c+2))^d,$$
where $\dist$ is measured with respect to the shortest path metric on $[2^n]^d$.
\end{lemma}

The bound we get from the lemma is the determining factor in our choice of $p$. Since each cube $Q$ is contained in the support of $\mu_k$ with probability $p$ and since $(2e (c+2))^d < 2^{4d}$ when $c$ is small, choosing $p=2^{-4d}$ means that there is a high probability that $\dist(Q_0, \supp \mu_k) \ge c d 2^{n-k}$.

\begin{proof}
The proof is a simple volume-packing argument. We consider $[2^n]^d$ as a subset of $\ell_1^d$, i.e., $\R^d$ equipped with the $\ell_1$--norm. The graph metric on $[2^n]^d$ is then the restriction of the $\ell_1$ metric on $\ell_1^d$. Let $V_d$ denote the $d$-dimensional volume measure on $\ell_1^d$ and for $r>0$ and $x\in \ell_1^d$, let $B_r(x)\subset \ell_1^d$ be the ball of radius $r$ around $x$. Then $V_d(B_r(x)) = \frac{r^d 2^d}{d!}$
(e.g., \cite[p.~3]{Bal97}), and since $d! \ge d^d e^{-d}$
(e.g., \cite[p.~17]{Spe14}),
  \begin{equation}\label{eq:ball-measure}
    V_d(B_r(x)) \le \left(2e\frac{r}{d}\right)^d.
  \end{equation}

  For any given $Q\in \D_k$, define $\hat{Q}\subset \ell^d_1$ by
  $$\hat{Q} = \{ q - x : q\in Q, x \in [0,1]^d\},$$
  so that $\hat{Q}$ is a cube in $\ell_1^d$ of side length $2^{n-k}$ that contains $Q$, and if $Q_1\ne Q_2$, then $\hat{Q}_1$ and $\hat{Q}_2$ have disjoint interiors.

  Fix a point $x_0\in Q_0$. Since $\diam \hat{Q} = d2^{n-k}$ for all $Q \in \D_k$, if $\dist(Q_0, Q)\le c d 2^{n-k}$, then
  $$\hat{Q} \subset B_{(cd+2d)2^{n-k}}(x_0).$$
  
  Since the interiors of the $\hat{Q}$'s are all disjoint and have volume $2^{d(n-k)}$,
  \begin{align*}
    \big|\{Q\in \D_k : \dist(Q_0, Q) \le & c d 2^{n-k}\}\big| \le \big|\{Q\in \D_k : \hat{Q} \subset B_{(cd+2d)2^{n-k}}(x_0)\}\big| \\
    &\le 2^{-d(n-k)} V_d(B_{(cd+2d)2^{n-k}}(x_0)).
  \end{align*}
  By \eqref{eq:ball-measure},
  $$|\{Q\in \D_k : \dist(Q_0, Q) \le cd 2^{n-k}\}| \le
  (2e (c+2))^d,$$
  as desired.
\end{proof}

\begin{theorem}\label{thm:TCnorm}
Let $k \in \{1,\dots n\}$ with $k \geq 24$. Then we have the lower bound $\E[\|\nu_k\|_{\TC}] \geq \frac{7}{48}d2^{n}$. 
\end{theorem}
\begin{proof}

  For each subset $\mcS \sbs \D_k$, we use the notation $\bigcup (\mcS) := \bigcup_{Q\in\mcS} Q$ and define the 1-Lipschitz function $f_\mcS: [2^n]^d \to [0,d2^{n}]$ by
  $$f_\mcS(z) = \min \left\{\dist\left(z,\bigcup (\mcS)\right), d 2^n\right\}.$$
  Note that $\diam [2^n]^d = d(2^n-1)$, so $f_\mcS(z) = \dist(z,\bigcup(\mcS))$ for each $S\sbs\D_k$, unless $\mcS = \emptyset$, in which case we have $f_\mcS \equiv d2^n$.
  
We shall use that
\begin{equation} \label{eq:||f||infty}
     \forall \mcS\sbs\D_k, \:\:\: \|f_\mcS\|_\infty \leq d2^n,
\end{equation}
and also the following consequence of Lemma \ref{lem:mu(grid)} for $k\ge 24$:
\begin{equation} \label{eq:mu(grid)}
\E[|\mu_k([2^n]^d)|] \leq \frac{1}{48}. 
\end{equation}

Let $\mcS_-(X)$ denote the random subset $\{Q\in\D_k: X_Q = -1\}$. Kantorovich duality implies that $\|\nu_k\|_{\TC}$ is almost surely lower bounded by the absolute value of the integral of the random 1-Lipschitz function $f_{\mcS_-(X)}$ against $\nu_k$. Using this, we get
\begin{align*}
    \E\|\nu_k\|_{\TC}
    &\geq \E\Big|\int f_{\mcS_-(X)} \d\nu_k\Big| \\
    &\geq \E\Big|\int f_{\mcS_-(X)} \d\mu_k\Big| - \E\Big|\int f_{\mcS_-(X)} \mu_k([2^n]^d) \dVol\Big| \\
    &\overset{\eqref{eq:||f||infty}}{\geq} \E\Big|\int f_{\mcS_-(X)} \d\mu_k\Big| - d2^n\E|\mu_k([2^n]^d)| \\
    &\overset{\eqref{eq:mu(grid)}}{\geq} \E\Big|\int f_{\mcS_-(X)} \d\mu_k\Big| - \frac{1}{48}d2^n.
\end{align*}
We will show that
$$\E \int f_{\mcS_-(X)} \d\mu_k \ge \frac{1}{6}d2^n,$$
from which the desired inequality will then follow.
  
When $\mcS \sbs \D_k$, and $Y$ is a random variable and $E$ is an event with $Y$ and $\one_E$ both functions of the random variables $\{X_Q\}_{Q\in\D_k}$, we use the conditioning notation
\begin{equation*}
\P_\mcS[E] = \frac{\P\left[\mcS_-(X)=\mcS \text{ and } E\right]}{\P\left[\mcS_-(X)=\mcS \right]} \text{\qquad}
\E_\mcS[Y] = \frac{\E\left[Y\cdot\one_{\{\mcS_-(X)=\mcS\}}\right]}{\P\left[\mcS_-(X)=\mcS \right]}.
\end{equation*}
Then
\begin{equation}\label{eq:cond-formula}
\E\int f_{\mcS_-(X)} \d\mu_k = \sum_{\mcS\sbs \D_k} \P\left[\mcS_-(X) = \mcS\right] \E_{\mcS} \int f_{\mcS_-(X)} \d\mu_k.
\end{equation}

We will bound the integral on the right by bounding it on every cube $Q\in \D_k$. Fix $\mcS \sbs \D_k$ and $Q \in \D_k$. Assume first that $Q \not\in \mcS$. If $\mcS_-(X) = \mcS$, then $X_Q \ne -1$, and thus since the $X_Q$'s are i.i.d., we have
$$\P_\mcS[X_Q=1] = \P[X_Q=1 | X_Q \ne -1] = \frac{p/2}{1-p/2} = \frac{p}{2-p}.$$
Define $p' := \frac{p}{2-p}$. Then
\begin{align}\label{eq:cond-exp-Q}
    \E_{\mcS}\int_Q f_{\mcS_-(X)} \d\mu_k &= p^{-1}2^k\E_{\mcS}\int_Q f_{\mcS_-(X)} X_Q\dVol \notag \\ 
    &= p^{-1}2^k\int_Q f_\mcS \frac{\E\left[X_Q \one_{\{\mcS_-(X) = \mcS\}}\right]}{\P[\mcS_-(X) = \mcS]}\dVol \notag \\
    &= p^{-1}2^k\int_Q f_\mcS \P_\mcS[X_Q=1]\dVol \notag \\
    &= p'p^{-1}2^k \int_{Q} f_\mcS \dVol.
\end{align}

On the other hand, if we assume that $Q \in \mcS$, then both sides of \eqref{eq:cond-exp-Q} are $0$. Therefore, \eqref{eq:cond-exp-Q} holds for all $Q$, and so summing over $Q\in\D_k$ yields
$$\E_{\mcS}\int f_{\mcS_-(X)} \d\mu_k = p'p^{-1}2^k \int f_\mcS \dVol.$$
Using \eqref{eq:cond-formula} and the above, we get
\begin{align} \label{eq:E[f_S]}
    \E\int f_{\mcS_-(X)} \d\mu_k &= \sum_{\mcS\sbs \D_k} \P\left[\mcS_-(X) = \mcS\right] \E_{\mcS} \int f_{\mcS_-(X)} \d\mu_k \notag \\
    &= \sum_{\mcS\sbs \D_k} \P\left[\mcS_-(X) = \mcS\right] p'p^{-1}2^k \int f_\mcS \dVol \notag \\
    &= p'p^{-1}2^k \E \int f_{\mcS_-(X)} \dVol \notag \\
    &\geq 2^{k-1} \E \int f_{\mcS_-(X)} \dVol.
\end{align}
We thus consider $\E\int f_{\mcS_-(X)} \dVol$.

Let $Q\in \D_k$. Then Lemma~\ref{lem:cardinality} (with $c=\frac{1}{2}$) states that there are at most $(5e)^d$ cubes $Q'\in \D_k$ such that $\dist(Q,Q') \le \frac{1}{2} d2^{n-k}$. If none of these cubes are elements of $\mcS_-(X)$, then $\min_{z\in Q} f_{\mcS_-(X)}(z) \ge \frac{1}{2} d2^{n-k}$, and so a union bound yields
\begin{equation*}
\P\left[\min_{z\in Q} f_{\mcS_-(X)}(z) \ge \frac{1}{2} d2^{n-k}\right] \ge 1 - \frac{p}{2}(5e)^d = 1 - \frac{1}{2}\left(\frac{5e}{16}\right)^d.
\end{equation*}
Since $d\ge 3$, we have that $1 - \frac{1}{2}(\frac{5e}{16})^d \ge \frac{2}{3}$, and thus\footnote{\label{foot:p=2^(-4d)} We chose $p=2^{-4d}$ in the previous section to obtain the estimate $1 - \frac{p}{2}(5e)^d \geq \frac{2}{3}$ on this line. If we chose $p=2^{-cd}$ for $c$ too small, then we would not have $1 - \frac{p}{2}(5e)^d > 0$. The choice $c=4$ is made to simplify numerics rather than to optimize constants.}
$$\E\int_Q f_{\mcS_-(X)} \dVol \ge \Vol(Q) \cdot \frac{1}{2} d2^{n-k} \cdot \frac{2}{3} = \frac{1}{3} d2^{n-k} \Vol(Q).$$
Summing over all $Q\in \D_k$, we find
$$\E\int f_{\mcS_-(X)} \dVol \ge \sum_{Q\in \D_k} \frac{1}{3} d2^{n-k} \Vol(Q) = \frac{1}{3} d2^{n-k}.$$
Therefore, by \eqref{eq:E[f_S]} and the above inequality,
$$\E\int f_{\mcS_-(X)} \d\mu_k \geq 2^{k-1} \E\int f_{\mcS_-(X)} \dVol \ge \frac{1}{6} d2^n,$$
as desired.
\end{proof}

\section{\texorpdfstring{Proving the Poincaré inequality}{Proving the Poincare inequality}} \label{sec:proof-Poincare}

For $f: [2^n]^d \to \R$, we define the $(1,1)$-Sobolev (semi)norm of $f$ by
$$\|f\|_{W^{1,1}} := \frac{1}{|E([2^n]^d)|}\sum_{\{u,v\}\in E([2^n]^d)} 2^n|f(v)-f(u)|.$$
In this section, we will prove the following Poincaré inequality for functions $[2^n]^d \to \R$ and see how it implies a lower bound for the $L_1$-distortion of $\TC([2^n]^d)$ and $\emd([2^n]^d)$.

\begin{theorem}[Poincaré inequality for $\nu_k$] \label{thm:Poincare}
There exists a constant $C<\infty$ (independent of $n,d$) such that, for any $f: [2^n]^d \to \R$, we have
\begin{equation*}
    \sum_{k=8}^n \E\left|\int f\d\nu_k\right| \leq C\|f\|_{W^{1,1}}.
\end{equation*}
\end{theorem}

We first prove Theorem~\ref{thm:Poincare} for indicator functions (Theorem~\ref{thm:setPoincare}) in \S\ref{sec:setPoincare}, then extend the proof to arbitrary functions in \S\ref{sec:Poincare}. The proof of Theorem~\ref{thm:Poincare} appears at the conclusion of this section.

\subsection{A Poincaré inequality for indicator functions} \label{sec:setPoincare}
For $A \sbs [2^n]^d$, define
\begin{equation*}
    \partial A := \{\{u,v\}\in E([2^n]^d): |\{u,v\} \cap A| = 1\}
\end{equation*}
and 
\begin{equation}\label{E:Per}
    \Per(A) := \frac{1}{d}2^{-n(d-1)}|\partial A|.
\end{equation}
Using $|E([2^n]^d)|=d(2^n)^{d-1}(2^n-1)$, we get
\begin{equation*} \label{eq:Per(A)bound}
    \Per(A) \leq 2^n\frac{|\partial A|}{|E([2^n]^d)|}.
\end{equation*}
The above inequality implies that
\begin{equation} \label{eq:Per(A)<=W1(A)}
    \Per(A) \leq \|\one_A\|_{W^{1,1}}.
\end{equation}

Bolloba\'s and Leader established in \cite{BL91} (edge) isoperimetric inequalities on high-dimensional grids. We state versions that are adapted to our purposes below.

\begin{theorem}[\cite{BL91}] \label{thm:isoperimetric}
There exists $C_{iso}\leq 2$ such that, for all $A \sbs [2^n]^d$,
\begin{equation}\label{E:IsoSmall}
    \Vol(A) \leq 2^{-2d} \implies \Vol(A)^{1-\frac1d} \leq C_{iso}\Per(A), \text{ and}
\end{equation}
\begin{equation}\label{E:IsoMedium}
    \Vol(A) \leq \frac12 \implies \Vol(A)^{1-\frac1d} \leq C_{iso}d\cdot\Per(A).
\end{equation}
\end{theorem}

\begin{proof}
{\bf Proof of \eqref{E:IsoSmall}.} 

By the definitions of $\Per(A)$ \eqref{E:Per} and $\Vol(A)$ \eqref{E:Vol}, inequality \eqref{E:IsoSmall} is equivalent to
\begin{equation}\label{E:IsoGraph}  |A| \le 2^{-2d} 2^{nd} \implies d|A|^{1-\frac1d} \le C_{iso} |\partial A|.
\end{equation}

Let $A \sbs [2^n]^d$. For $r \in [1,d]$, define $f_r: [0,\infty) \to [0,\infty)$ by
$$f_r(t) := t^{1-\frac1r}\cdot r\cdot(2^n)^{\frac dr-1}.$$
Bolloba\'s-Leader \cite[Theorem 3]{BL91} proved that
\begin{equation}\label{E:BLTh3}|\partial A|\ge\min\{f_r(|A|):~r\in\{1,\dots d\}\}.
\end{equation}

Suppose that $|A| = 2^{-kd} 2^{nd}$ with $k \ge 2$. Then
\[f_r(|A|)=2^{-kd(1-\frac1r)}\cdot r\cdot (2^n)^{d-1},\]
and thus
\[\frac{d}{dr} f_r(|A|)=2^{-kd(1-\frac1r)}\cdot (2^n)^{d-1}\left(-\frac{kd}{r^2}\ln 2\cdot r+1\right).\]

If $r \le d$, then $r\le 2d\ln2\le kd\ln 2$, so the function $r \mapsto f_r(|A|)$ is decreasing for $r\in [1,d]$. Thus the infimum in \eqref{E:BLTh3} is attained for $r=d$, and we have
\begin{equation}
  |\partial A|\ge d|A|^{1-\frac1d}.  
\end{equation}
That is, \eqref{E:IsoGraph} holds for any $C_{iso} \geq 1$.
\bigskip

{\bf Proof of \eqref{E:IsoMedium}:} 
Using the definitions of $\Per(A)$ \eqref{E:Per} and $\Vol(A)$ \eqref{E:Vol}, inequality \eqref{E:IsoMedium} is equivalent to 
\begin{equation}\label{E:IsoGraphM}  |A| \le 2^{nd - 1} \implies |A|^{\frac{d-1}{d}} \le C_{iso} |\partial A|.
\end{equation}
By \eqref{E:IsoSmall}, it suffices to consider the case that 
\begin{equation}\label{E:Case2} 2^{-2d}\cdot 2^{nd}\le |A|\le \frac{1}{2}2^{nd}.\end{equation}

We use the following parts of \cite[Corollary 4]{BL91}:

\begin{equation}\label{E:BLCor4}
	|\partial A|\ge \begin{cases}
		\frac{4|A|}{2^n} &\hbox{ if } |A|< \frac{2^{nd}}4,\\
		2^{n(d-1)} &\hbox{ if } \frac{2^{nd}}4\le |A|\le \frac{3\cdot 2^{nd}}4.
			\end{cases}
\end{equation}

If 
\[\frac14\cdot {2^{nd}}\le |A|\le \frac12\cdot 2^{nd},\]
then \eqref{E:IsoGraphM} (with any $C_{iso}\ge 1$) immediately follows from the second part of \eqref{E:BLCor4}. 
We thus consider the case
\[\frac1{2^{m+1}}\cdot {2^{nd}}\le |A|< \frac1{2^m}\cdot 2^{nd}, \quad 2\le m< 2d-1.\]
In this case, the first inequality in \eqref{E:BLCor4} implies
\[	|\partial A|\ge \frac{4}{2^n}\cdot \frac1{2^{m+1}}\cdot 2^{nd}=\frac1{2^{m-1}}\cdot 2^{n(d-1)}.\]
On the other hand, since $2\le m< 2d-1$,
\[|A|^{\frac{d-1}d}<\frac1{2^{m\frac{d-1}d}}\cdot 2^{n(d-1)} \le \frac2{2^{m-1}}\cdot 2^{n(d-1)}.\]
If $2\le m< 2d-1$, then 
\[\frac1{2^{m-1}}\ge\frac12\cdot\frac1{2^{m\frac{d-1}d}},
\]
so $|A|^{\frac{d-1}{d}}\le 2|\partial A|$. Therefore, \eqref{E:IsoGraphM} holds for any $C_{iso} \ge 2$.
\end{proof}

The next two lemmas prove the Poincaré inequality for indicator functions for each of the two summands in the definition of $\nu_k$ (see \eqref{E:DefNu}). The first one provides a sufficient bound for $\mu_k(A)$, $A \sbs [2^n]^d$. In the proofs, the requirement $d \geq 3$ is again essential.

\begin{lemma} \label{lem:Poincare1-alt}
 Let $k \in \N$ with $k \geq 8$. For any $A \sbs [2^n]^d$ with $\Vol(A) \leq \frac12$, let 
 $m(A) \in \N$ be such that $2^{-m(A)} < \Vol(A)^{\frac{1}{d}} \leq 2^{-m(A)+1}$. Then
\begin{equation*}
  \E[|\mu_k(A)|] \leq 2^{9}C_{iso}2^{-\frac{|m(A)-k|}{2}}\Per(A).
\end{equation*}
\end{lemma}
\begin{proof}
  For brevity, set $m=m(A)$. Set $N := 2^{kd}\Vol(A) =\Vol(A)/\Vol(Q)$ for $Q\in \D_k$. We claim that
 \begin{equation}\label{eq:mu-k-bound}  \E[|\mu_k(A)|]
     \le 2^k \Vol(A) \min \{1, (pN)^{-\frac{1}{2}}\}.
  \end{equation}
  To gain intuition for this inequality, assume for the moment that $A$ is equal to the union of some number of cubes $Q$ in $\D_k$. Then $N$ is the number of such cubes, and thus $\mu_k(A)$ is an i.i.d.\ sum of $N$ symmetric random variables taking the value 0 with probability $1-p$ and the two values $\pm p^{-1}2^k\frac{\Vol(A)}{N}$ with probability $p/2$ each. It is easy to see that \eqref{eq:mu-k-bound} holds for such an i.i.d.\ sum. We will prove that the same inequality holds, in fact, for any $A \sbs [2^n]^d$.
  
  First, for every $Q\in \D_k$ and $x\in Q$, the expected density of $|\mu_k|$ with respect to $\Vol$ at $x$ is given by
  \begin{equation}\label{eq:mu-density}
    \E\Big[\Big|\frac{\d\mu_k}{\dVol}(x)\Big|\Big] = \E\Big[\Big|\frac{\mu_k(Q)}{\Vol(Q)}\Big|\Big] = \E[|p^{-1}2^kX_Q|] = 2^k.
  \end{equation}
  Therefore,
  \begin{equation}\label{eq:mu-k-bound-1}
    \E[|\mu_k(A)|] = \E\Big[\Big|\int_A \frac{\d\mu_k}{\dVol}\dVol\Big|\Big]
    \leq \int_A\E\Big[\Big|\frac{\d\mu_k}{\dVol}\Big|\Big]\dVol
    = 2^k\Vol(A).
  \end{equation}
  This bound is sharp when $A\sbs Q$ for some $Q\in \D_k$.

  Second, since the random variables $\{X_Q\}_{Q\in\D_k}$ are independent mean-zero variables with $\E[X_Q^2] = p$ for all $Q\in\D_k$, we can express the variance of $\mu_k(A)$ in terms of the intersections $Q\cap A$:
  \begin{equation*}
    \E[\mu_k(A)^2] = \sum_{Q\in\D_k} \E\left[(p^{-1}2^k X_Q\Vol_Q(A))^2\right] = p^{-2}2^{2k}\sum_{Q\in\D_k}\E[X_Q^2]\Vol_Q(A)^2 = p^{-1}2^{2k}\sum_{Q\in\D_k}\Vol_Q(A)^2.
  \end{equation*}
  This is a powerful identity, because $\Vol(A) = \sum_Q \Vol_Q(A)$ and $\Vol_Q(A) \le \Vol(Q) = 2^{-kd}$. In particular,
  \begin{equation*}
    \E[\mu_k(A)^2] \leq p^{-1}2^{2k} \max_{Q\in \D_k} \Vol_Q(A) \sum_{Q\in\D_k} \Vol_Q(A) \le p^{-1} 2^{2k} 2^{-kd} \Vol(A).
  \end{equation*}
  Therefore,
  \begin{align}\label{eq:mu-k-bound-2}
\notag    \E[|\mu_k(A)|] &\le \sqrt{\E[\mu_k(A)^2]} \\
\notag     &\le \sqrt{p^{-1} 2^{2k} 2^{-kd} \Vol(A)} \\ 
\notag     &= 2^k \Vol(A) \sqrt{p^{-1}2^{-kd}\Vol(A)^{-1}} \\
    &= 2^k \Vol(A)(pN)^{-\frac{1}{2}}.
  \end{align}
  Inequality \eqref{eq:mu-k-bound} now follows from \eqref{eq:mu-k-bound-1} and \eqref{eq:mu-k-bound-2}.

  To prove the lemma, we consider two cases. First, suppose that $k-m \le 5$. 
  Since $k\ge 8$, we have $m\ge 3$, so $\Vol(A) \le 2^{-2d}$.
  By our choice of $m$ and $N$,
  \begin{equation}\label{eq:k-m}
    2^{(k-m-4)d} < pN \le 2^{(k-m-3)d}.
  \end{equation}  
  Using the first part of Theorem~\ref{thm:isoperimetric}, we have
  \begin{equation*}
    \frac{\E[|\mu_k(A)|]}{\Per(A)} \stackrel{\eqref{eq:mu-k-bound-1}}{\le} 2^k \frac{\Vol(A)}{\Per(A)} \overset{\eqref{E:IsoSmall}}{\le} C_{iso} \Vol(A)^{\frac{1}{d}} 2^k \le C_{iso} 2^{k-m+1}.
  \end{equation*}
  The ratio
  $$\frac{2^{k-m+1}}{2^{-\frac{|k-m|}{2}}} = 2^{k-m+1 + \frac{|k-m|}{2}}$$
  is increasing with $k-m$, so $2^{k-m+1} < 2^9 2^{-\frac{|k-m|}{2}} $, and thus
  $$\frac{\E[|\mu_k(A)|]}{\Per(A)} \le C_{iso} 2^{k-m+1} \le 2^9 C_{iso} 2^{-\frac{|k-m|}{2}},$$
  as desired.
  
  Otherwise, suppose that $k-m \ge 6$. Using the second part of Theorem~\ref{thm:isoperimetric}, we have
  \begin{align*}
    \frac{\E[|\mu_k(A)|]}{\Per(A)} \stackrel{\eqref{eq:mu-k-bound-2}}{\le} 2^k \frac{\Vol(A)}{\Per(A)} (pN)^{-\frac{1}{2}} \overset{\eqref{E:IsoMedium}}{\le} 2^k C_{iso} d \Vol(A)^{\frac{1}{d}} (pN)^{-\frac{1}{2}} \le d C_{iso} 2^{k-m+1} (pN)^{-\frac{1}{2}}.
  \end{align*}
  Then, by \eqref{eq:k-m},
  \begin{equation*}
    \frac{\E[|\mu_k(A)|]}{\Per(A)} \le d C_{iso} 2^{k-m+1} 2^{-\frac{d}{2}(k -m) + 2d} = 2^{2d + 1 - \frac{d-3}{2}(k-m)}d C_{iso} 2^{-\frac{k-m}{2}}.
  \end{equation*}
  Since $d\ge 3$ and $k-m\ge 6$, we have that
  $$2^{2d + 1 - \frac{d-3}{2}(k-m)} d \le 2^{2d + 1 - 3(d-3)} d = 2^{- d + 10} d \le 2^9.$$
  Thus, $\E[|\mu_k(A)|] \le 2^9 C_{iso} \Per(A) 2^{-\frac{|k-m|}{2}}$, which proves the lemma.
\end{proof}

The next lemma provides a sufficient bound for the other summand in $\nu_k(A)$, $\mu_k([2^n]^d)\Vol(A)$. It is a simple consequence of Lemma~\ref{lem:mu(grid)}.

\begin{lemma} \label{lem:Poincare2}
Let $k \in \N$ with $k \geq 5$. For any $A \sbs [2^n]^d$ with $\Vol(A) \leq \frac12$,
$$\E[|\mu_k([2^n]^d)\Vol(A)|] \leq (\sqrt{2})^{16-k}C_{iso}\Per(A).$$
\end{lemma}

\begin{proof}
Let $A \sbs [2^n]^d$ with $\Vol(A) \leq \frac12$. By Theorem~\ref{thm:isoperimetric} and the fact that $\Vol(A)\leq \Vol(A)^{1-\frac1d}$, it suffices to prove that
\begin{equation*}
    \E[|\mu_k([2^n]^d)\Vol(A)|] \leq (\sqrt{2})^{16-k}\frac{1}{d}\Vol(A).
\end{equation*}
This inequality is precisely the conclusion of Lemma~\ref{lem:mu(grid)}, with each side multiplied by $\Vol(A)$.
\end{proof}

\begin{theorem}[Poincaré for Indicator Functions] \label{thm:setPoincare}
 There exists a constant $C<\infty$ (independent of $n,d$) such that, for any $A \sbs [2^n]^d$,
\begin{equation*}
    \sum_{k=8}^n \E[|\nu_k(A)|] \leq C\Per(A).
\end{equation*}
\end{theorem}

\begin{proof}
Let $A \sbs [2^n]^d$. Since each side of the desired inequality is unchanged by the permutation $A \leftrightarrow A^c$, we may assume that $\Vol(A) \leq \frac12$. Let $m(A) \in \N$ with $2^{-m(A)} < \Vol(A)^{\frac1d} \leq 2^{-m(A)+1}$. Then using Lemmas~\ref{lem:Poincare1-alt} and \ref{lem:Poincare2}, we can find constants $C_1,C_2,C < \infty$ (independent of $n,d,A$) such that
\begin{align*}
    \sum_{k=8}^n &\E[|\nu_k(A)|] \leq \sum_{k=8}^n \E[|\mu_k(A)|] + \E[|\mu_k([2^n]^d)\Vol(A)|] \\
    &< \sum_{k=8}^\infty C_1(\sqrt{2})^{-|m(A)-k|}\Per(A) + C_2(\sqrt{2})^{-k}\Per(A) \leq C\Per(A).
\end{align*}
\end{proof}

\subsection{From indicator functions to general functions} \label{sec:Poincare}

The following theorem allows one to reduce the proof of a Poincaré inequality from general functions to indicator functions. 

\begin{theorem}[{\cite[Theorem 2.3]{GO26}}] \label{thm:set-to-function}
Let $|||\cdot|||$ be a seminorm on the set of functions from $[2^n]^d \to \R$ such that $|||\cdot|||$ vanishes on constant functions, and let $C \in [0,\infty)$ be a constant. If $|||\one_A||| \leq C \|\one_A\|_{W^{1,1}}$ for all subsets $A \sbs [2^n]^d$, then $|||f||| \leq C\|f\|_{W^{1,1}}$ for all $f: [2^n]^d \to \R$.
\end{theorem}

\begin{proof}[Proof of Theorem~\ref{thm:Poincare}]
The theorem follows immediately from Theorems~\ref{thm:setPoincare} and \ref{thm:set-to-function} and \eqref{eq:Per(A)<=W1(A)}.
\end{proof}

\section{\texorpdfstring{Poincaré inequality implies $L_1$-distortion}{Poincaré inequality implies L1-distortion}} \label{sec:Poincare-distortion}

In this final section, we deduce our main results Theorem~\ref{thm:mainthm} and Corollary~\ref{cor:maincor}. To do so, we need the following linearization result, which is a consequence of Bourgain's discretization theorem  \cite{Bou87} (see \cite[\S5]{GO26} for details).

\begin{theorem}[{\cite[\S5]{GO26}}]\label{T:lin&EMD}
For every finite metric space $X$, the equality
\[c_1(\emd(X))=c_1(\tc(X)) = c_{1,{\rm lin}}(\tc(X))\] holds, where $c_{1,{\rm lin}}(\tc(X))$ denotes the infimal distortion among all $M\in\N$ and all {\it linear} maps $f: \tc(X) \to \ell^{M}_1$.
\end{theorem}

We also require the following fact. Suppose that $X$ is a finite metric space, $\lambda: \TC(X) \to \R$ is a linear functional, and $x_0\in X$ is a basepoint. Let $f(x) = \lambda(\delta_x-\delta_{x_0})$. Then, since $\{\delta_x-\delta_{x_0}\}_{x\in X}$ spans $\TC(X)$, for any $\mu \in \TC(X)$, we have $\lambda(\mu)=\int fd\mu$.

The next theorem and corollary establish Theorem~\ref{thm:mainthm} and Corollary~\ref{cor:maincor} in the case $d \geq 3$. As previously discussed, the case $d=2$ already follows from \eqref{eq:lowerbound}.

\begin{theorem} \label{thm:L1distortion}
For $n,d\in\N$ with $d \geq 3$, there exists a constant $C < \infty$ (independent of $n,d$) such that
$c_1(\TC([2^n]^d)) \geq C^{-1}(n-23)d$.
\end{theorem}

\begin{proof}
By Theorem~\ref{T:lin&EMD}, it suffices to prove that $c_{1,\rm lin}(\TC([2^n]^d)) \geq C^{-1}(n-23)d$. Let $L: \TC([2^n]^d): \to \ell_1^{M}$ be an arbitrary noncontractive linear operator into a finite-dimensional $\ell_1$-space. Choose functions $\{f_i: [2^n]^d \to \R\}_{i=1}^{M}$ such that, for all $\mu \in \TC([2^n]^d)$, $\int f_i \d\mu = L(\mu)_i$ for all $i\in\{1,\dots M\}$. Then, applying Theorem~\ref{thm:Poincare} to each $f_i$, we get
\begin{equation*}
    \sum_{k=24}^n \E\left|L(\nu_k)_i\right| \leq \frac{C}{|E([2^n]^d)|}\sum_{\{u,v\}\in E([2^n]^d)}2^n|L(\delta_v-\delta_u)_i|.
\end{equation*}
Summing over $i \in \{1,\dots M\}$ yields
\begin{equation*}
    \sum_{k=24}^n \E\left\|L(\nu_k)\right\|_1 \leq \frac{C}{|E([2^n]^d)|}\sum_{\{u,v\}\in E([2^n]^d)}2^n\|L(\delta_v-\delta_u)\|_1.
\end{equation*}
Since $\|\delta_v-\delta_u\|_{\TC} = 1$, this yields
\begin{equation*}
    \sum_{k=24}^n 2^{-n}\E\left\|L(\nu_k)\right\|_1 \leq C\|L\|_{op},
\end{equation*}
where $\|L\|_{op}$ denotes the operator norm. Then since $L$ is noncontractive, this implies
\begin{equation*}
    \sum_{k=24}^n 2^{-n}\E\left\|\nu_k\right\|_{\TC} \leq C\|L\|_{op}.
\end{equation*}
By Theorem~\ref{thm:TCnorm}, the left-hand side is lower bounded by $(n-23)\frac{7d}{48}$, completing the proof.
\end{proof}

The lower bound in Theorem~\ref{thm:L1distortion} easily self-improves, allowing $n - 23$ to be replaced with $n$ (with a larger constant $C$).

\begin{corollary} \label{cor:L1distortion}
For $n,d\in\N$ with $d \geq 3$, there exists a constant $C < \infty$ (independent of $n,d$) such that
$c_1(\TC([2^n]^d)) \geq C^{-1}nd$.
\end{corollary}

\begin{proof}
If $n \geq 24$, then the conclusion follows from Theorem~\ref{thm:L1distortion}, so assume that $1 \leq n \leq 23$. This means we only need to show that
\begin{equation} \label{eq:NTS}
    c_1(\TC([2^n]^d)) \gtrsim d.
\end{equation}

The graph $[2^n]^{2^{24}}$ contains a geodesic of length $2^{24}$, meaning it contains an isometrically embedded path subgraph $P_{2^{24}}$. Hence, for any $d' \in \N$, the graph $[2^n]^{2^{24}d'} = ([2^n]^{2^{24}})^{d'}$ contains an isometrically embedded grid subgraph $[2^{24}]^{d'}$. By Theorem~\ref{thm:L1distortion}, this results in
$$c_1(\TC([2^n]^{2^{24}d'})) \geq c_1(\TC([2^{24}]^{d'})) \gtrsim (24-23)d' = d',$$
from which \eqref{eq:NTS} obviously follows.
\end{proof}

\section*{Acknowledgments}
We would like to express our gratitude to Assaf Naor, who suggested this problem to us and organized a meeting on it during the conference ``Metric Embeddings" at the American Institute of Mathematics (July 2025). We thank all participants of that meeting as well.

\author{Chris Gartland}, 
{Department of Mathematics and Statistics, University of North Carolina at Charlotte, University City Blvd, Charlotte, NC 28223, U.S.A.}

\author{Mikhail Ostrovskii},
{Department of Mathematics and Computer Science, St. John's University, 8000 Utopia Parkway, 
Queens, NY 11439, USA}

\author{Yuval Rabani},
{The Rachel and Selim Benin School of Computer Science and Engineering, The Hebrew University of Jerusalem, 9190416 Jerusalem, Israel}

\author{Robert Young},
{Courant Institute School of Mathematics, Computing, and Data Science, New York University, 251 Mercer St., New York, NY 10012, USA}

\end{document}